\documentclass{amsart}
\usepackage{amsmath,amsfonts,amssymb,amsthm}
\usepackage[utf8]{inputenc}
\usepackage[T1]{fontenc}
\usepackage{a4wide}
\usepackage{enumitem}
\usepackage{multirow}
\usepackage{booktabs}
\usepackage{url}
\usepackage[dvipsnames]{xcolor}
\usepackage{comment}
\excludecomment{comment}
\usepackage{fourier}
\DeclareMathAlphabet{\mathcalzp}{OMS}{zplm}{m}{n}
\usepackage{tikz,tikz-cd}
\usetikzlibrary{arrows,arrows.meta}
\tikzcdset{arrow style=tikz, diagrams={>=latex}}
\usetikzlibrary{calc}

\specialcomment{commentDRK}{\begingroup \color{MidnightBlue}}{\endgroup}
\specialcomment{remarkDRK}{\begingroup \color{Aquamarine}\it }{\endgroup}

\newcommand{\cC}{\mathcal{C}}

\newcommand{\cE}{\mathcal{E}}
\newcommand{\cF}{\mathcal{F}}
\newcommand{\cO}{\mathcal{O}}

\newcommand{\cX}{\mathcal{X}}
\newcommand{\cY}{\mathcal{Y}}
\newcommand{\cW}{\mathcal{W}}

\renewcommand{\AA}{\mathbb{A}}

\newcommand{\FF}{\mathbb{F}}
\newcommand{\GG}{\mathbb{G}}
\newcommand{\MM}{\mathbb{M}}

\newcommand{\PP}{\mathbb{P}}
\newcommand{\QQ}{\mathbb{Q}}
\newcommand{\ZZ}{\mathbb{Z}}

\newcommand{\XC}{\mathcal{X}}

\newcommand{\GL}{\mathrm{GL}}
\newcommand{\SL}{\mathrm{SL}}
\newcommand{\PSL}{\mathrm{PSL}}

\newcommand{\isom}{\cong}
\newcommand{\dsp}{\displaystyle}

\newcommand{\ff}{\mathfrak{f}}
\newcommand{\uu}{\mathfrak{u}}

\newcommand{\Group}{G}

\newcommand{\ccomma}{\raisebox{0.4ex}{,}}

\newcommand{\mmu}{\boldsymbol{\mu}}

\newtheorem{theorem}{Theorem}
\newtheorem{corollary}[theorem]{Corollary}

\newtheorem{proposition}[theorem]{Proposition}

\newtheorem*{theorem*}{Theorem}
\newtheorem*{definition*}{Definition}

\makeatletter
\newcommand{\leqnomode}{\tagsleft@true\let\veqno\@@leqno}
\newcommand{\reqnomode}{\tagsleft@false\let\veqno\@@eqno}

\makeatother

\makeatletter
\newcommand*{\relrelbarsep}{.264ex}
\newcommand*{\relrelbar}{%
  \mathrel{%
    \mathpalette\@relrelbar\relrelbarsep
  }%
}
\newcommand*{\@relrelbar}[2]{%
  \raise#2\hbox to 0pt{$\m@th#1\relbar$\hss}%
  \lower#2\hbox{$\m@th#1\relbar$}%
}
\providecommand*{\rightrightarrowsfill@}{%
  \arrowfill@\relrelbar\relrelbar\rightrightarrows
}
\providecommand*{\leftleftarrowsfill@}{%
  \arrowfill@\leftleftarrows\relrelbar\relrelbar
}
\providecommand*{\xrightrightarrows}[2][]{%
  \ext@arrow 0359\rightrightarrowsfill@{#1}{#2}%
}
\providecommand*{\xleftleftarrows}[2][]{%
  \ext@arrow 3095\leftleftarrowsfill@{#1}{#2}%
}
\makeatother

\begin{document}

\title{Weber modular curves and modular isogenies}
\author{Leonardo Col\`o and David Kohel}
\date{}
\reqnomode

\begin{abstract}
  We study the modular curves defined by Weber functions, and associated
  modular polynomials, action of $\SL_2(\ZZ)$, and parametrizations of
  elliptic curves with a view to the study of the isogeny graphs that
  they determine, particularly for supersingular elliptic curves.
  In addition to applications to efficient isogeny computation in
  cryptographic applications, we present an application to explicit
  Galois representations.
\end{abstract}

\maketitle

\section{Introduction}

We describe explicit models of modular curves defined by Weber functions,
with their parametrizations of families of elliptic curves.
These functions, appearing in Weber~\cite{Weber}, are modular functions
of level $48$, each of which describe a degree-$72$ covers of the
$j$-line by a genus-$0$ modular curve.
This degree is the largest known of any modular function covering the
$j$-line, and the properties of the resulting class polynomials and
modular polynomials make them attractive for explicit algorithms for
class groups and isogeny computations.  Collectively the functions
generate the function field of a high-genus quotient of $X(48)$
which admits many symmetries. 

Various cryptographic algorithms rely on the explicit computation of
modular isogenies: given an elliptic curve $E/\FF_q$ or its $j$-invariant
$j(E)$ and a prime $\ell$, one needs to determine one or all $\ell+1$
of the moduli of $\ell$-isogenous elliptic curves.
The data of an $\ell$-isogeny can be specified by the $x$-coordinates
of points in an $\ell$-torsion
subgroup~\cite{Velu71}, corresponding to a point on $X_1(\ell)$, or a
kernel polynomial vanishing on these points (see \cite[Chapter 2]{Kohel96}
and \cite{LercierMorain98}) associated to a point on $X_0(\ell)$.
Conversely, one can recover the isogeny data on $E$ from the associated
point on $X_0(\ell)$~\cite{Schoof95}.

Yui and Zagier~\cite{YuiZagier97} investigated Weber functions for
constructing small class invariants. These class polynomials exhibited
remarkably small coefficient size, asymptotically 72 times smaller
in height.
Gee~\cite{Gee99} developed the theoretical foundations of class invariants
for arbitrary modular functions, using Shimura reciprocity, and Enge and
Morain~\cite{EngeMorain02,EngeMorain09} investigated the height reduction
in terms of Dedekind eta and generalized Weber functions.
To date, no modular function is known which achieves the same height reduction,
and the factor 72 is conjecturally maximal.

In this work, we describe the modular polynomials, curves, and elliptic
curve parametrizations defined by the Weber functions.
The modular polynomials have the same advantage of height reduction,
as for class polynomials, in addition to a sparseness condition,
which makes them much smaller and efficient to use for isogeny computations. 
The curves defined by the triple of classical Weber functions (with a
suitable twist), which are skew conjugate under the action of $\SL_2(\ZZ)$.
We describe the position of this curve and its quotients in the Galois
cover from $X(48)$ to $X(1)$.
In particular, a quotient to the $8$-th Fermat curve gives an alternative
curve of level $16$ admitting a degree-$48$ cover of $X(1)$.
We provide explicit models for these curves, their morphisms and the
action of $\SL_2(\ZZ)$, as well as the elliptic curve parametrizations
and isogenies that they describe.

The level structure determined by Weber functions is particularly suited
to traversal of supersingular isogeny graphs.
The SIDH protocol~\cite{DeFeoJao} works with $j$-invariants of level~$1$,
but the CSIDH protocol already incorporates the level-$4$ structure of
a function $a$ on $X_1(4)$ such that $j = 256(a^2-3)^3/(a^2-4)$.
In addition to admitting small modular polynomials, the supersingular
points on the Weber curves split completely over $\FF_{p^2}$, exactly
as for the $j$-invariants.  No further base extension is required to
split the supersingular isogenies or to define their isogenies.

In addition to cryptographic applications, the structure of supersingular
isogeny graphs permits one to compute representations associated to spaces
of modular forms, in which the level structure of the invariants used,
and the characteristic of the base field $\FF_{p^2}$, controls the level
of the representations.
Mestre's method of graphs~\cite{Mestre86} gives a construction for Galois
representations of prime level $p$ or level $m p$ where $m$ is one
of the levels $m$ in $\{2,3,4,5,7,13\}$ for which the modular curve 
$X_0(m)$ is of genus $0$.  We conclude with examples of this application
of Weber curves to the study of Galois representations, particularly
associated to modular elliptic curves with prescribed ramification
at $2$ and $3$.

\section{Weber functions and modular polynomials}
\label{Section:WeberModuli}

Let $\XC_\Group/\FF_p$ be a modular curve associated to level $N$ open subgroup
$\Group \subseteq \GL_2(\hat{\ZZ})$, and $\ell$ prime to $N$.
The modular $\ell$-isogeny supersingular graph is determined by a set of supersingular
moduli points on $\XC_\Group$, with edge relations given by points in the correspondence,
$$
\XC_\Group(B_0(\ell)) \xrightrightarrows[\hspace{5mm}]{\hspace{5mm}} \XC_\Group.
$$
When $\ell$ is not coprime to $N$, we replace the above correspondence with
$$
\XC_\Group(B_0(\ell^t)) \xrightrightarrows[\hspace{5mm}]{\hspace{5mm}} \XC_\Group,
$$
where $t$ is the minimal integer such that $\Group \not\subseteq B_0(\ell^t)$.

One advantage of working with a modular curve $\XC_\Group$ of higher level, still of genus $0$,
is that the logarithmic coefficient size of the modular polynomials $\Phi_\ell(x,y)$ with
respect to a degree one function is reduced by a factor of the degree $\XC_\Group \to X(1)$.
If we consider the smallest correspondence, for $\ell = 2$, for the $j$-invariant we
have the modular polynomial
$$
\begin{array}{c}
  x^3 - x^2 y^2 + y^3 + 1488 x^2 y + 1488 x y^2 - 162000 x^2 - 162000 y^2 + 40773375 x y\\
  +\ 8748000000 x + 8748000000 y - 157464000000000.
\end{array}
$$
In comparison, for the group $\Gamma_0(3)$ of index $4$, this becomes:
$$
x^3 - x^2 y^2 - 24 x^2 y - 24 x y^2 - 729 x y + y^3
$$
and for the full congruence subgroup $\Gamma(3)$ of index $12$ we have:
$$
x^3 - x^2 y^2 + 9 x y + y^3 - 54
$$
It is worth noting that the sparsity of monomials of the latter polynomial
can be explained by the transformation 
$$
\Phi_\ell(\zeta_3 x,\zeta_3^\ell y) = \zeta_3^{\ell+1} \Phi_\ell(x,y),
$$
of the family of modular polynomials with respect to the particular normalized
modular function on $\XC_\Group = \XC(3)$.  In particular, only the monomials $x^iy^j$
satisfying $i + \ell j \equiv \ell+1 \bmod 3$ can occur.
In the next section we derive similar results for the modular polynomials
defined in terms of Weber functions.

\subsection*{Weber modular polynomials}

The best known reduction in coefficient size as well as in sparsity of coefficients
is obtained for the Weber function $\ff$ of level~$48$,
$$
\ff(\tau) = \zeta_{48}^{-1} \frac{\eta\left(\frac{\tau+1}{2}\right)}{\eta(\tau)},
$$
which generates a degree-$72$ cover of the $j$-line, given by
$$
j = \frac{(\ff^{24}-16)^3}{\ff^{24}}.
$$
The modular polynomials with respect to $\ff$ are the integral polynomials
$\Phi_\ell(x,y)$ such that
$$
\Phi_\ell(\ff(\tau),\ff(\ell\tau)) = 0.
$$
Although the Weber function does not generate the full modular curve $X(48)$,
which has genus $2689$, it still satisfies a transformation giving the following
symmetry of its induced modular polynomials.

\begin{proposition}
  \label{prop:Weber_trans}
  The modular polynomial $\Phi_\ell(x,y)$ of prime level $\ell$ with respect to the Weber
  function satisfies the transformation:
  $$
  \Phi_\ell(\zeta_{24} x,\zeta_{24}^\ell y) = \zeta_{24}^{\ell+1} \Phi_\ell(x,y),
  $$
  with respect to a primitive $24$-th root of unity $\zeta_{24}$.
\end{proposition}

\noindent
This gives the following sparsity result for the coefficients of the Weber modular polynomials.

\begin{corollary}
  \label{cor:Weber_sparcity}
  The coefficient of the monomial $x^iy^j$ in the Weber modular polynomial
  $\Phi_\ell(x,y)$ is nonzero only if $i+\ell j \equiv \ell+1 \bmod 24$.
\end{corollary}

Asymptotically, modular polynomials have $(\ell+1)^2$ monomials, but due
to the sparseness of the Weber polynomials the number of nonzero
coefficients is of the order of $(\ell+1)^2/24$.
Combined with the height reduction of the coefficients this makes the
Weber modular polynomials attractive for constructing isogeny invariants.
The first few examples are as follows.
$$
\begin{array}{l}
\Phi_5(x,y) = x^6 - x^5 y^5 + 4 x y + y^6\\
\Phi_7(x,y) = x^8 - x^7 y^7 + 7 x^4 y^4 - 8 x y + y^8\\
\Phi_{11}(x,y) = x^{12} - x^{11} y^{11} + 11 x^9 y^9 - 44 x^7 y^7 + 88 x^5 y^5 - 88 x^3 y^3 + 32 x y + y^{12}\\
\Phi_{13}(x,y) = x^{14} - x^{13} y^{13} + 13 x^{12} y^2 + 52 x^{10} y^4 + 78 x^8 y^6 + 78 x^6 y^8 + 52 x^4 y^{10} + 13 x^2 y^{12} + 64 x y + y^{14}
\end{array}
$$
Going further, the modular polynomial $\Phi_{71}(x,y)$ has exactly $3\cdot 71 = 213$
nonzero coefficients, ignoring the symmetry $\Phi_\ell(x,y) = \Phi_\ell(y,x)$, which
implies an even smaller number of distinct coefficients.

In the interest of constructing $\ell$-isogeny chains, especially for $\ell = 2$
or $\ell = 3$, we note that the $48$-level structure gives the modular polynomials
$\Phi_2(x,y)$ and $\Phi_3(x,y)$ a particular form.  We descend the $2$-level
structure by setting $t = -\ff^8$, so that
$$
j = \left(\frac{t^3+16}{t}\right)^3
$$
With respect to this function, we obtain the modular polynomial: 
$$
\Psi_2(x,y) = (x^2 - y)y + 16x
$$
and the Weber modular polynomial $\Phi_2(x,y) = -\Psi_2(-x^8,-y^8)$ remains
irreducible.\footnote{
More correctly, the modular polynomial $\Psi_2(x,y)$ satisfies
$\Psi_2(\ff_1^8(\tau),\ff_1^8(2\tau)) = 0$, where $\ff_1$ is the conjugate Weber function
$$
\ff_1^8(\tau) = -\ff(\tau+3)^8 = \left(\frac{\eta(\frac{\tau}{2})}{\eta(\tau)}\right)^8
$$
and hence $\Psi_2(\ff^8(\tau),\ff^8(2\tau-3)) = 0$.  Nevertheless, this modular relation
describes a $2$-isogeny relation of the underlying curves, extending the parametrized
$2$-isogeny to a $4$-isogeny, and will be used for defining our $2$-isogeny chains.}
A similar descent of the $3$-level to the function $r = \ff^3$, gives the modular polynomial
$$
\Psi_3(x,y) = x^4 - x^3 y^3 + 8 x y + y^4,
$$
such that $\Psi_3(r(\tau),r(3\tau)) = 0$, for which $\Phi_3(x,y) = \Psi_3(x^3,y^3)$
is irreducible.  For a given supersingular Weber invariant, these relations determine
orbits under multiplication by $\zeta_8$ or $\zeta_3$, but in view of the global
relation of Proposition~\ref{prop:Weber_trans}, the lift to the orbit can be chosen to be
compatible with isogeny relations of other prime degrees.

\subsection*{Weber modular functions}
The definition of a modular polynomial requires a genus~$0$ modular curve or
fixed function on a modular curve.  In order to understand the transformation
properties of this function, one needs to study its images under the images
of transformations by the generators of $\SL_2(\ZZ)$.  

The classically defined triple of Weber functions,
$$
\ff(\tau) = \zeta_{48}^{-1} \frac{\eta\left(\frac{\tau+1}{2}\right)}{\eta(\tau)},
\quad
\ff_1(\tau) = \frac{\eta\left(\frac{\tau}{2}\right)}{\eta(\tau)},
\quad
\ff_2(\tau) = \sqrt{2}\cdot\frac{\eta(2\tau)}{\eta(\tau)},
$$
are modular functions, the first two with rational integral $q$-expansions
in $q^{1/48} = e^{2\pi{i}\tau/48}$, and the third with $q$-expansion in 
$q^{1/24}$ with coefficients in $\sqrt{2}\ZZ$.
Each satisfy the relations
$$
j = \frac{(\ff^{24}-16)^3}{\ff^{24}}
  = \frac{(\ff_1^{24}+16)^3}{\ff_1^{24}}
  = \frac{(\ff_2^{24}+16)^3}{\ff_2^{24}}\ccomma
$$
on which the generators $S$ and $T$ of $\SL_2(\ZZ)$ induce (see Yui and
Zagier~\cite{YuiZagier97}, Gee~\cite{Gee99}):
\begin{equation}
\label{Weber_transform:Gee}
(\ff,\ff_1,\ff_2) \circ S = (\ff,\ff_2,\ff_1)  \mbox{ and }
(\ff,\ff_1,\ff_2) \circ T = (\zeta_{48}^{-1}\ff_1,\zeta_{48}^{-1}\ff,\zeta_{48}^2\ff_2).
\end{equation}
It is well-known that the Weber functions satisfy
$
\ff^8 = \ff_1^8 + \ff_2^8,
$
and based on the identity
$$
\zeta_{48}^{-1}\eta\left(\frac{\tau+1}{2}\right)\eta\left(\frac{\tau}{2}\right)\eta(2\tau) = \eta(\tau)^3
$$
it follows from the definition of the Weber functions that $\ff\, \ff_1 \ff_2 = \sqrt{2}$.

Setting
$(\uu_0,\uu_1,\uu_2) = (\ff,\zeta_{16}\ff_1(\tau),\zeta_{16}^{-1}\ff_2(\tau))$
the triple $(\uu_0,\uu_1,\uu_2)$ satisfies the common relations
$$
j = \frac{(\uu_0^{24}-16)^3}{\uu_0^{24}} = \frac{(\uu_1^{24}-16)^3}{\uu_1^{24}} = \frac{(\uu_2^{24}-16)^3}{\uu_2^{24}}\ccomma
$$
and one verifies that the three orbits $\{ \zeta_{24}^j \uu_i \;:\; j \in \ZZ/24\ZZ \}$, for $i \in \ZZ/3\ZZ$, 
run over the $72$ roots of the modular polynomial
$$
(x^{24} - 16)^3 - j(q) x^{24}
$$
in $\QQ(\zeta_{48})[\![q^{1/48}]\!]$.  The proposition which follows summarizes the action
of the modular group on the normalized Weber functions.

\begin{proposition}
  \label{prop:Weber_group_action}
The action of $\PSL_2(\ZZ) = \langle{S,T}\rangle$ on Weber triples $(\uu_0,\uu_1,\uu_2)$
maps through the quotient $\SL_2(\ZZ/48\ZZ)/\{\pm1\}$, and is defined on generators by:
$$
(\uu_0,\uu_1,\uu_2) \circ S = (\uu_0,\zeta_{8}\uu_2,\zeta_{8}^{-1}\uu_1)  \mbox{ and }
(\uu_0,\uu_1,\uu_2) \circ T = (\zeta_{12}^{-1}\uu_1,\zeta_{24}\uu_0,\zeta_{24}\uu_2).
$$
In particular the elements $U = T^{-1}ST$, $V = T^{-2}ST^2$, and $W = ST^3$ act by permutations:
$$
(\uu_0,\uu_1,\uu_2) \circ U = (\uu_2,\uu_1,\uu_0),\quad
(\uu_0,\uu_1,\uu_2) \circ V = (\uu_0,\uu_2,\uu_1),\quad
(\uu_0,\uu_1,\uu_2) \circ W = (\uu_2,\uu_0,\uu_1).
$$
\end{proposition}

\begin{proof}
  The quotient via $\SL_2(\ZZ/48\ZZ)$ follows from the definition of the Weber functions,
  as $\eta$ quotients of level $48$.
  From the definition $(\uu_0,\uu_1,\uu_2) = (\ff,\zeta_{48}^3\ff_1(\tau),\zeta_{48}^{-3}\ff_2(\tau))$
  and the transformations \eqref{Weber_transform:Gee}, we obtain:
$$
  (\uu_0,\uu_1,\uu_2) \circ S 
  = (\ff,\zeta_{48}^3\ff_2,\zeta_{48}^{-3}\ff_1) 
  = (\uu_0,\zeta_{48}^6\uu_2,\zeta_{48}^{-6}\uu_1),
$$
and
$$
  (\uu_0,\uu_1,\uu_2) \circ T 
  = (\zeta_{48}^{-1}\ff_1,\zeta_{48}^2\ff,\zeta_{48}^{-1}\ff_2)
  = (\zeta_{24}^{-2}\uu_1,\zeta_{24}\uu_0,\zeta_{24}\uu_2).
$$
The permutation actions of $U$, $V$, and $W$ follow from the actions of $S$ and $T$.
\end{proof}

\noindent{\bf N.B.}
The right action on Weber triples gives a homomorphism $\iota: \PSL_2(\ZZ) \to \GL_3(\QQ(\zeta_{24}))$,
defined on generators by
$$
\iota(S)
= \left(\begin{array}{@{}ccc@{}}
1 & 0 & 0\\
0 & 0 & \zeta_{8}^{-1}\\
0 & \zeta_{8} & 0
\end{array}\right)\ccomma
\quad
\iota(T)
= \left(\begin{array}{@{}ccc@{}}
0 & \zeta_{24} & 0\\
\zeta_{24}^{-2} & 0 & 0\\ 
0 & 0 & \zeta_{24}
\end{array}\right)\cdot
$$
In the next section we describe the image as the automorphism group of a Galois cover $\cW_{24} \to X(1)$
of the $j$-line, of order $1152$ whose kernel $\Gamma_{24}$ contains $\PSL_2(\ZZ/48\ZZ)$ of index $32$.
Since $\PSL_2(\ZZ)$ acts by skew permutations, its image $G = \iota(\PSL_2(\ZZ))$ lies in the subgroup
$S_3 \ltimes \mmu_{24}^3 \subset \GL_3(\QQ(\zeta_{24}))$, and there exists an exact sequence
$$
1 \longrightarrow D \longrightarrow G \longrightarrow S_3 \longrightarrow 1.
$$
where $D = G \cap \mmu_{24}^3$ is the subgroup of diagonal matrices.  By the existence of the permutation
subgroup $S_3 = \langle{U,V,W}\rangle \subset G$, the sequence splits, from which the corollary follows.

\begin{corollary}
The subgroup $D$ of diagonal matrices in $G = \iota(\PSL_2(\ZZ))$ is a normal subgroup,
generated by the order $24$ matrices:
$$
\iota(T)^2
= \left(\begin{array}{@{}ccc@{}}
      \zeta_{24}^{-1}  &  0      &   0\\
         0  &   \zeta_{24}^{-1}  &   0\\
         0  &        0   & \zeta_{24}^{2}
\end{array}\right)
\quad\mbox{and}\quad
\iota(STS)^2 = \left(\begin{array}{@{}ccc@{}}
      \zeta_{24}^{-1}  &  0       &   0\\
         0 & \zeta_{24}^{2} &   0\\
         0        &  0   & \zeta_{24}^{-1}
\end{array}\right)\ccomma
$$
subject to the relation $\iota(T)^{16} = \iota(STS)^{16} = \zeta_3^{-1} I$.
In particular $D \subset \mmu_{24}^3 \subset \GL_3(\QQ(\zeta_{24}))$ is
abelian of order $3 \cdot 8^2 = 192$.  The quotient $G/D$ is isomorphic to
the symmetric group $S_3$, acting by permutation on the Weber functions
$\{\uu_0^{24},\uu_1^{24},\uu_2^{24}\}$, and $G \isom S_3 \ltimes D$.
\end{corollary}

In the following section we consider the embeddings given by the normalized Weber functions,
and their quotients by subgroups $D[m] = D \cap \mmu_m^3 \subset \GL_3(\QQ(\zeta_{24}))$.

\section{Weber curves}

We are now able to use the conjugate Weber functions to define a curve with
projective embedding given by the functions $\uu_0$, $\uu_1$, and $\uu_2$.
Let $m$ and $n$ be integers such that $mn = 8$.
The map determined by the normalized Weber functions $(\uu_0^m:\uu_1^m:\uu_2^m:1)$
determines a {\it Weber modular curve} $\cW_{3n}$ in $\PP^3$
\begin{equation}
  \label{eqn:weber3n}
\cW_{3n}: \left\{
\begin{array}{l}
  X_0^n + X_1^n + X_2^n = 0,\\
  X_0 X_1 X_2 = \sqrt{2}^m X_3^3
\end{array}\right.
\end{equation}
with quotient Weber curve $\cW_{n}$ defined
as the image of $(\uu_0^{3m}:\uu_1^{3m}:\uu_2^{3m}:1)$ in $\PP^3$:
\begin{equation}
  \label{eqn:weber1n}
\mbox{\hspace{6mm}}
\cW_{n}: \left\{
\begin{array}{l}
  X_0^n + X_1^n + X_2^n = 48 X_3^n,\\
  X_0 X_1 X_2 = \sqrt{8}^{m} X_3^3.
\end{array}\right.
\end{equation}
These defining relations follow directly from the relations $\ff^8 = \ff_1^8 + \ff_2^8$
and $\ff\,\ff_1\ff_2 = \sqrt{2}$, and the curves are equipped with maps
$\cW_{mn} \to \cW_{n}$ for each product $mn$ dividing $24$.

The curves $\cW_n$ form Galois covers of the $j$-line $X(1)$.  In order to define modular
polynomials, as correspondences in $\PP^1 \times \PP^1$, we will work with the quotients
$\pi_i:\cW_n \to \cX_n \isom \PP^1$ determined by the projection $(u_0:u_1:u_2:1)
\mapsto (u:1) = (u_i:1)$, and equipped with the map to the $j$-line $\cX_n \to X(1)$
given by
$$
j(u) = \frac{(u^{n}-16)^3}{u^{n}}\cdot
$$
In particular, the curve $\cX_n$ parametrizes the family of elliptic curves
\begin{equation}
  \label{eqn:E0/Xn}
\cE_0/\cX_n: y^2 = x\left(x^2 - \frac{(u^n - 64)}{4}x - (u^n - 64)\right)\ccomma
\end{equation}
of $j$-invariant $j(u)$, discriminant $(u^n-64)^3u^{n}$, and $2$-torsion point $(0,0)$.
The quotient curve is
\begin{equation}
  \label{eqn:E1/Xn}
\cE_1/\cX_n: y^2 = x\left(x^2 + \frac{(u^n - 64)}{2}x + \frac{(u^n - 64)}{16}u^n\right)
\end{equation}
with $j$-invariant $-(u^n-256)^3/u^{2n}$, discriminant $-(u^n-64)^3u^{2n}$, and $2$-torsion
point $(0,0)$.

By eliminating the function $u_2$, it is clear that the cover $\cW_n \to \cX_n$ is of
degree $2n$, and the expression for $j$ shows that $\cX_n \to X(1)$ is of degree $3n$.
This gives a degree $6n^2$ cover $\cW_n \to X(1)$.  In the next section we determine
the structure of the Galois group of this cover. 

We conclude this section with an observation of an isomorphism, for $n$ dividing $8$,
between the Weber curve $\cW_n$ and the Fermat curve $\cF_n : X^n + Y^n + Z^n = 0$.
In particular the cover $\cW_{3n} \to \cW_n$ factors through the quotient
$$
\begin{tikzcd}[row sep=0pt]
  \cW_{3n} \arrow[r] & \cF_n. \\
  (X_0:X_1:X_2:X_3) \arrow[|->,r] & (X_0:X_1:X_2)
\end{tikzcd}
$$
Indeed this is the quotient by the group of automorphisms
$$
\{ (X_0:X_1:X_2:X_3) \mapsto (\zeta_3^iX_0:\zeta_3^iX_1:\zeta_3^iX_2:X_3) \;:\; i \in \ZZ/3\ZZ \},
$$
which stabilizes the fibers of the quotient $\cW_{3n} \to \cW_n$, which must factor
through $\cF_n$. In view of the defining equations for $\cW_{3n}$, we infer that
the induced map $\cF_n \to \cW_n$ is given by:
$$
(X:Y:Z) \longmapsto \left(X^3:Y^3:Z^3:\frac{XYZ}{\sqrt{2}^m}\right).
$$
In what follows, we prove that $\cW_{3n} \to \cW_n$ has degree~$3$, hence $\cF_n \to \cW_n$
is an isomorphism.
As with the Weber curves, the Fermat curves come equipped with a triple of projections
$\pi_i: \cF_n \to \cY_n \isom \PP^1$:
$$
\pi_i((X:Y:Z)) = \left\{
\begin{array}{l}
  (Y:Z) \mbox{ if } i = 0,\\
  (X:Z) \mbox{ if } i = 1,\\
  (X:Y) \mbox{ if } i = 2,
\end{array}\right.
$$
parametrized by the modular functions $(\uu_1/\uu_2)^m$, $(\uu_0/\uu_2)^m$ and $(\uu_0/\uu_1)^m$,
respectively. This gives the diagram of morphisms between Weber and Fermat curves and their quotients:
$$
\begin{tikzcd}[column sep=10mm]
  \cX_8 \arrow[d,"\ZZ/2\ZZ",swap]   && \arrow[ll,swap,"\ZZ/2\ZZ\times\ZZ/8\ZZ"] \cW_8 \arrow[d,"(\ZZ/2\ZZ)^2"] \arrow[rr,"\isom"] && \cF_8 \arrow[d,"(\ZZ/2\ZZ)^2",swap] \arrow[rr,"\ZZ/8\ZZ"] && \cY_8 \arrow[d,"\ZZ/2\ZZ"] \\
  \cX_4 \arrow[d,"\ZZ/2\ZZ",swap]   && \arrow[ll,swap,"\ZZ/2\ZZ\times\ZZ/4\ZZ"] \cW_4 \arrow[d,"(\ZZ/2\ZZ)^2"] \arrow[rr,"\isom"] && \cF_4 \arrow[d,"(\ZZ/2\ZZ)^2",swap] \arrow[rr,"\ZZ/4\ZZ"] && \cY_4 \arrow[d,"\ZZ/2\ZZ"] \\
  \cX_2 \arrow[d,"\ZZ/2\ZZ",swap]   && \arrow[ll,swap,"\ZZ/2\ZZ\times\ZZ/2\ZZ"] \cW_2 \arrow[d,"(\ZZ/2\ZZ)^2"] \arrow[rr,"\isom"] && \cF_2 \arrow[d,"(\ZZ/2\ZZ)^2",swap] \arrow[rr,"\ZZ/2\ZZ"] && \cY_2 \arrow[d,"\ZZ/2\ZZ"] \\
  \cX_1 \arrow[ddrrr,bend right=12,end anchor=west] &&
  \cW_1 \arrow[ll,swap,"\ZZ/2\ZZ"] \arrow[ddr,bend right=16,end anchor={[yshift=1.0mm]west}] \arrow[rr,"\isom"] &&
  \cF_1 \arrow[ddl,bend left=16,end anchor={[yshift=1.0mm]east}] \arrow[rr,"\isom"] &&
  \cY_1 \arrow[ddlll,bend left=12] \\[-5mm]
  u \arrow[ddrrr,mapsto,bend right=16,end anchor=west] && && && s \arrow[ddlll,mapsto,bend left=12,end anchor=east]\\[-2mm]
  &&& X(1) \\[-7mm]
  &&& \displaystyle j = \frac{(u-16)^3}{u} = 2^8\frac{(s^2+s+1)^3}{s^2(s+1)^2}
\end{tikzcd}
$$
The isomorphisms $\cW_n \to \cF_n$ appear in the appendix, for all divisors $n$ of $8$. In particular,
for $n = 1$, the isomorphism is given by the following map:
\begin{equation}
  \label{eqn:Weber1-Fermat1-isom}
\begin{tikzcd}[row sep=-2mm,column sep=8mm]
  \cW_1: \left\{
  \begin{array}{@{}c@{}}
    X_0 + X_1 + X_2 = 48 X_3,\\
    X_0 X_1 X_2 = 4096 X_3^3
  \end{array} \right. \arrow[r] & \cF_1 : X + Y + Z = 0\\[1mm]
  (u_0:u_1:u_2:u_3) \arrow[r,mapsto] & (u_0-16u_3 : u_1-16u_3 : u_2-16u_3)\\[1mm]
  (16s_0^3:16s_1^3:16s_2^3:s_0s_1s_2) & \arrow[l,mapsto] (s_0:s_1:s_2)
\end{tikzcd}
\end{equation}
The affine parametrization $(s:t:1)$ of $\cF_1$ by $s = -(t+1)$ gives the parametrization
of $\cW_1$:
$$
  (u_0:u_1:u_2:1)
  = 
  \left( -\frac{16s^2}{(s+1)}: \frac{16(s+1)^2}{s}:  -\frac{16}{s(s+1)}: 1\right)
  = 
  \left( \frac{16(t+1)^2}{t}: -\frac{16t^2}{(t+1)}:  -\frac{16}{t(t+1)}: 1\right).
$$
After the substitution $u = -16s^2/(s+1)$ and quadratic twist of the models $\cE_0/\cX_1$ and $\cE_1/\cX_1$, we obtain
the isogenous families
\begin{equation}
  \label{eqn:C0/Y1-C1/Y1}
\cC_0/\cY_1: y^2 = x(x - 1)(x + t) = x(x - 1)(x - (s + 1))
\mbox{ and }
\cC_1/\cY_1: y^2 = x((x + s)^2 + 4x).
\end{equation}
of $j$-invariants $256(s^2+s+1)^3/s^2(s+1)^2$ and $16(s^2+16s+16)^3/s^4(s+1)$, and discriminants
$16s^2(s+1)^2$ and $256s^4(s+1)$, respectively.

\subsection*{Modular group}

To each factorization $mn = 24$, the Weber curve $\cW_{n}$ in $\PP^3$, defined by
the triple of Weber functions $(\uu_0^m,\uu_1^m,\uu_2^m)$, comes equipped with an
action of $\PSL_2(\ZZ)$. We denote the kernel of the action by $\Gamma_n$, identifying
the Weber curves with the modular curve $X(\Gamma_n)$.  The action of $\PSL_2(\ZZ)$
on Weber functions induces a representation in $\GL_3(\QQ(\zeta_n))$ determined by
the images of the generators $S$ and $T$.
$$
\begin{tikzcd}[row sep=1mm]
  \PSL_2(\ZZ) \arrow[r] & \GL_3(\QQ(\zeta_n)) \\
  S, T \arrow[r,mapsto] &
  \left(\begin{array}{*{5}{c@{\;}}}
    1 & 0 & 0 \\
    0 & 0 & \overline{\zeta}_{8}^{m} \\ 
    0 & \zeta_{8}^{m} & 0
  \end{array} \right)\ccomma
  \left(\begin{array}{*{5}{c@{\;}}}
    0 & \zeta_{24}^m & 0 \\
    \overline{\zeta}_{24}^{2m} & 0 & 0 \\
    0 & 0 & \zeta_{24}^m \end{array}\right)
\end{tikzcd}
$$
The image group is finite, whose kernel $\Gamma_n$ is a normal congruence subgroup
of $\PSL_2(\ZZ)$, such that $\PSL_2(\ZZ)/\Gamma_n$ is the Galois group of the cover
$\cW_n \to X(1)$.

We reduce the computation of $\Gamma_n$ by first proving that
$\Gamma_1 \isom \Gamma(2)$, then noting that $\Gamma_1 \subset \Gamma_3 \cap \Gamma_8
= \Gamma_{24}$, which reduces to determining $\Gamma_3$ and the chain
$\Gamma_1 \subset \Gamma_2 \subset \Gamma_4 \subset \Gamma_8$.

Let $\mmu_m = \langle\zeta_m\rangle$ be the group of $m$-th roots of unity. We identify
$\mmu_m^3$ with the subgroup (scheme)
of $\GG_m^3$, acting on affine space $\AA^3$.
We define the antidiagonal group,
$$
\nabla(\mmu_m^2) = \left\{
  \big(\zeta_m^i,\zeta_m^j,\zeta_m^{-i-j}\big) \in \mmu_m^3 \;|\; i,j \in \ZZ/m\ZZ \right\},
$$
and the diagonal group
$
\Delta(\mmu_m) = \left\{ \big(\zeta_m^i,\zeta_m^i,\zeta_m^i\big) \in \mmu_m^3 \;|\; i \in \ZZ/m\ZZ \right\}.
$
If $m \equiv 0 \bmod 3$, then $\Delta(\mmu_3) \subset \nabla(\mmu_m^2)$, otherwise
the groups $\Delta(\mmu_m)$ and $\nabla(\mmu_m)$ are independent.

\begin{proposition}
  \label{prop:Weber_groups}
  For each divisor $mn$ of $24$, the morphism $\cW_{mn} \to \cW_n$ is the quotient by
  a subgroup of automorphisms in $\mmu_m^3$, isomorphic to $\Gamma_{n}/\Gamma_{mn}$.
  \begin{itemize}
  \item
    If $m = 3$, then the automorphism group of the cover is $\Delta(\mmu_3)$. 
  \item
    If $m$ divides $8$, then the automorphism group of the
    cover is $\nabla(\mmu_m^2)$. 
  \end{itemize}
  In particular, the degree of $\cW_{3n} \rightarrow \cW_n$ is $3$, and if $m$ divides $8$,
  the degree of $\cW_{mn} \rightarrow \cW_n$ is $m^2$.
\end{proposition}

\begin{proof}
  The quotients $\cW_{mn} \to \cW_n$ are induced by the restriction, to the
  affine Weber curves, of the quotients $\AA^3 \to \AA^3$ by $\mmu_m^3$,
  sending $(x_0,x_1,x_2)$ to $(x_0^m,x_1^m,x_2^m)$.  Each Weber curve is a
  normal cover of $X(1)$, and the relative Galois group of the cover
  $\cW_{mn} \to \cW_n$ is
  $
  \left(\PSL_2(\ZZ)/\Gamma_{n}\right)/\left(\PSL_2(\ZZ)/\Gamma_{mn}\right) \isom \Gamma_{n}/\Gamma_{mn}.
  $

  We first consider $m = 3$ : the affine Weber curve $\cW_{3n}$ ($X_3 \ne 0$)
  is given by equations $x^{n}+y^{n}+z^{n} = 0$ and $x y z = c_{3n} \ne 0$,
  with $n$ coprime to $3$.
  Consequently the subgroup of $\mmu_3^3$ stabilizing the curve is the
  intersection of the group $\Delta(\mmu_3)$, stabilizing $x^n + y^n + z^n$,
  and $\nabla(\mmu_3)$, fixing $xyz$.
  Since $\Delta(\mmu_3) \subset \nabla(\mmu_3)$, the automorphism group of
  the cover is $\Delta(\mmu_3)$, and $\cW_{3n} \to \cW_n$ is of degree~$3$.

  Now we consider $m$ divides $8$. The affine Weber curve is defined by 
  equations $x^{mn} + y^{mn} + z^{nm} = 48$, if $n$ is coprime to $3$,
  or otherwise $x^{mn/3} + y^{mn/3} + z^{nm/3} = 0$, and $xyz = c_{mn} \ne 0$.
  The former is fixed by $\mmu_m^3$, hence the automorphism group of the
  cover is the subgroup $\nabla(\mmu_m)$ fixing the form $xyz$.
  The degrees of the morphisms $\cW_{mn} \to \cW_n$ follows since
  $|\Delta(\mmu_3)| = 3$ and $|\nabla(\mmu_m)| = m^2$.
\end{proof}

\begin{proposition}
  \label{prop:Weber_group1}
  The Weber kernel group $\Gamma_1$ equals $\Gamma(2)$ and $\cW_1 \isom X(2)$.
\end{proposition}

\begin{proof}
  The Weber curve $\cW_2$ is parametrized by the three normalized $\eta$-quotients:
  $$
  \uu_0^{24} = -{\left(\!\frac{\eta((\tau+1)/2)}{\eta(\tau)}\!\right)\ccomma}^{\!\!\!\!24}\quad
  \uu_1^{24} = -{\left(\!\frac{\eta(\tau/2)}{\eta(\tau)}\!\right)\ccomma}^{\!\!\!\!24}\quad
  \uu_2^{24} = -{2^{12}\left(\!\frac{\eta(2\tau)}{\eta(\tau)}\!\right)\ccomma}^{\!\!\!\!24}
  $$
  invariant under $\Gamma(2)$, which generate the $S_3 = \PSL_2(\FF_2)$-extension $\QQ(X(2))/\QQ(X(1))$.
\end{proof}

\begin{proposition}
  \label{prop:Weber_group3}
  The Weber kernel group $\Gamma_3$ equals $\Gamma(2) \cap \Gamma_{ns}^+(3)$,
  and for each $n$ dividing $8$
  $$
  \Gamma_{3n} = \Gamma_n \cap \Gamma_{ns}^+(3).
  $$
\end{proposition}

\begin{proof}
  By Proposition~\ref{prop:Weber_group1} we have $\Gamma_1 = \Gamma(2)$ and
  by Proposition~\ref{prop:Weber_groups}, we have $\Gamma_3/\Gamma_1 \isom \Delta(\mmu_3)$
  of order $3$. Thus $\Gamma_3$ is a congruence subgroup satisfying
  $
  \Gamma(6) \subset \Gamma_3 \subset \Gamma(2),
  $
  of index $3$ in $\Gamma(2)$. This uniquely characterizes $\Gamma_3$ as
  the intersection of $\Gamma(2)$ with the normal subgroup
  $\Gamma_{ns}^+(3)$ of $\PSL_2(\ZZ)$.
\end{proof}

\noindent
It thus suffices to characterize the groups $\Gamma_n$ for $n$ dividing $8$.

\begin{proposition}
  \label{prop:Weber_group2}
  The Weber kernel group $\Gamma_2$ equals $\Gamma(4)$ and $\cW_2 = X(4)$.
\end{proposition}

\begin{proof}
  By Proposition~\ref{prop:Weber_groups}, the Weber kernel group of $\cW_2 \to \cW_1$
  satisfies $\Gamma_2/\Gamma_1 \isom \Delta(\mmu_2) \isom (\ZZ/2\ZZ)^2$, where
  $\Gamma_1 = \Gamma(2)$ by Proposition~\ref{prop:Weber_group1}.
  From the transformation of the $\eta$-quotients defining the Weber functions, it is
  clear that the triple $(\uu_0^{12},\uu_1^{12},\uu_2^{12})$ are invariant under $\Gamma(4)$.
  Since $\Gamma(4)/\Gamma(2) \isom (\ZZ/2\ZZ)^2$, it follows that $\Gamma_2 = \Gamma(4)$.
\end{proof}

\begin{proposition}
  \label{prop:Weber_group4}
  The Weber kernel group $\Gamma_4$ equals $\Gamma_s(8)$ and $\cW_4 = X_s(8)$.
\end{proposition}

\begin{proof}
  The triple $(\uu_0^{6},\uu_1^{6},\uu_2^{6})$ 
  is invariant
  under $\Gamma(8)$, but the quotient $\Gamma(4)/\Gamma(8)$ is isomorphic to
  $(\ZZ/2\ZZ)^3$, so that $\Gamma_4$ is an index~$2$ subgroup.
  An explicit calculation with the above images of the generators, shows that
  $\Gamma_4$ contains the diagonal subgroup, isomorphic to $(\ZZ/8\ZZ)^*$,
  hence equals $\Gamma_s(8)$.
\end{proof}

It remains to characterize the group $\Gamma_8$ under which the triple of functions
$(\uu_0^3,\uu_1^3,\uu_2^3)$ is invariant. This group is not the split Cartan subgroup
$\Gamma_s(16)$, but we can show that
$$
\Gamma(16) \subset \Gamma_8 \subset \Gamma_s(8)
= \left\{\;\left(\begin{array}{@{}cc@{}}a&b\\c&d\end{array}\right)
  \;:\; b \equiv c \equiv 0 \bmod 8 \right\},
$$
and that the group $\Gamma_8/\Gamma(16)$ is cyclic of order $4$ generated by
$$
T^2 U^2 T^{-2} U^{-2} \equiv \left(\begin{array}{@{}cc@{}}13&8\\8&5\end{array}\right) \bmod 16,
$$
where $U = STS^{-1}$.
The equality is easily verified in $\SL_2(\ZZ/16\ZZ)$ and the word expression on the
left maps to the identity under the above homomorphism to $\GL_3(\QQ(\zeta_{8}))$,
showing that the element is in the kernel of the action on $\cW_8$.  Moreover, the
matrix on the right lifts to $\SL_2(\ZZ)$.
Given that the degree of $\cW_8 \to \cW_4 = X_s(8)$ is $4$, and $X(16) \to X(8)$
is of degree $16$, this proves the following description of the kernel group $\Gamma_8$.

\begin{proposition}
  \label{prop:Weber_group8}
  The Weber kernel group $\Gamma_8$ is the group generated by $\Gamma(16)$ and
  $\left(\begin{array}{@{}cc@{}}13&8\\8&5\end{array}\right)$.
\end{proposition}

\noindent
In particular we note that $\Gamma_8$ contains the diagonal matrix in the center
of $\SL_2(\ZZ/16\ZZ)$:
$$
\left(\begin{array}{@{}cc@{}}13&8\\8&5\end{array}\right)^2 \equiv
\left(\begin{array}{@{}cc@{}}9&0\\0&9\end{array}\right) \in \SL_2(\ZZ/16\ZZ)
$$
hence contains the subgroup
$$
\Gamma(16,8,16) = \left\{
  \left(\begin{array}{@{}cc@{}}a&b\\c&d\end{array}\right) \in \SL_2(\ZZ) \;:\;
  a \equiv d \equiv 1 \bmod 8, b \equiv c \equiv 0 \bmod 16 \right\}. 
$$
Given that $\Gamma_s(8)/\Gamma(16)$ is an abelian group:
$$
\Gamma_s(8)/\Gamma(16) =
\left\langle\left(\begin{array}{@{}cc@{}}13&0\\0&5\end{array}\right)\ccomma
\left(\begin{array}{@{}cc@{}}1&8\\0&1\end{array}\right)\ccomma
\left(\begin{array}{@{}cc@{}}1&0\\8&1\end{array}\right)
\right\rangle \isom C_4 \times V_4,
$$
so that $\Gamma_s(8)/\Gamma(16,8,16) \isom C_2 \times V_4 = C_2^3$; we have
the following diagram of modular curves between $X(16)$ and $X_s(8)$, where
the curves represented by dots 
are intermediate quotients by the subgroups of $V_4 = \langle{T^8,U^8}\rangle$.
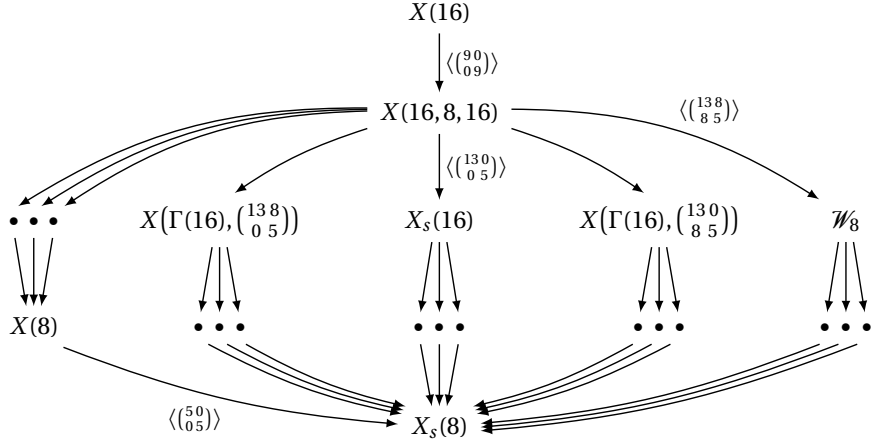
$$
\begin{tikzcd}[row sep=8mm] 
  & & X(16) \arrow[d,"\big\langle\binom{\,9\;0\,}{\,0\;9\,}\big\rangle"]\\
  & & X(16,8,16)
  \arrow[dll,bend right=16, end anchor={[xshift=-1mm]north east}]
  \arrow[dll,bend right=16, end anchor={[xshift=1mm]north}]
  \arrow[dll,bend right=16, end anchor={[xshift=3mm]north west}]
  \arrow[dl,bend right=8,end anchor={[xshift=2mm]north}]
  \arrow[d,"\big\langle\binom{\,13\;0\,}{\,\;0\;\;5\,}\big\rangle"]
  \arrow[dr,bend left=8,end anchor={[xshift=-2mm]north}]
  \arrow[drr,bend left=16,"\big\langle\binom{\,13\;8\,}{\,\;8\;\;5\,}\big\rangle",end anchor=north west] \\
\bullet\;\bullet\;\bullet
  \arrow[d,start anchor={[xshift=-2.5mm]south}, end anchor={[xshift=-1mm]north}]
  \arrow[d,start anchor={south}]
  \arrow[d,start anchor={[xshift=+2.5mm]south}, end anchor={[xshift=+1mm]north}]
  & X\big(\Gamma(16),\binom{\,13\;8\,}{\,\;0\;\;5\,}\big)
  \arrow[d,end anchor={[xshift=-2.5mm]north},start anchor={[xshift=-1mm]south}]
  \arrow[d]
  \arrow[d,end anchor={[xshift=2.5mm]north},start anchor={[xshift=1mm]south}]
  & X_s(16)
  \arrow[d,end anchor={[xshift=-2.5mm]north},start anchor={[xshift=-1mm]south}]
  \arrow[d]
  \arrow[d,end anchor={[xshift=2.5mm]north},start anchor={[xshift=1mm]south}]
  & X\big(\Gamma(16),\binom{\,13\;0\,}{\,\;8\;\;5\,}\big)
  \arrow[d,end anchor={[xshift=-2.5mm]north},start anchor={[xshift=-1mm]south}]
  \arrow[d]
  \arrow[d,end anchor={[xshift=2.5mm]north},start anchor={[xshift=1mm]south}]
  & \cW_8
  \arrow[d,end anchor={[xshift=-2.5mm]north},start anchor={[xshift=-1mm]south}]
  \arrow[d]
  \arrow[d,end anchor={[xshift=2.5mm]north},start anchor={[xshift=1mm]south}]\\
X(8) \arrow[drr,bend right=8,"\big\langle\binom{\,5\;0\,}{\,0\;5\,}\big\rangle" swap,start anchor={[xshift=-1mm]south east}]
  & \bullet\ \bullet\ \bullet
  \arrow[dr,bend right=8,start anchor={[xshift=-1.5mm]south},end anchor={[xshift=+0.50mm,yshift=-1mm]north west}]
  \arrow[dr,bend right=8,start anchor={[xshift=+1.0mm]south},end anchor={[xshift=+0.75mm,yshift=-0.5mm]north west}]
  \arrow[dr,bend right=8,start anchor={[xshift=+3.5mm]south},end anchor={[xshift=+1.00mm]north west}]
  & \bullet\ \bullet\ \bullet
  \arrow[d,start anchor={[xshift=-2.5mm]south},end anchor={[xshift=-1mm]north}]
  \arrow[d]
  \arrow[d,start anchor={[xshift=+2.5mm]south},end anchor={[xshift=+1mm]north}]
  & \bullet\ \bullet\ \bullet
  \arrow[dl,bend left=8,start anchor={[xshift=-3.5mm]south},end anchor={[xshift=-1mm]north east}]
  \arrow[dl,bend left=8,start anchor={[xshift=-1.0mm]south},end anchor={[xshift=-0.75mm,yshift=-0.5mm]north east}]
  \arrow[dl,bend left=8,start anchor={[xshift=+1.5mm]south},end anchor={[xshift=-0.50mm,yshift=-1mm]north east}]
  & \bullet\ \bullet\ \bullet
  \arrow[dll,bend left=8,start anchor={[xshift=-3.5mm]south},end anchor={[yshift=+0.5mm]east}]
  \arrow[dll,bend left=8,start anchor={[xshift=-1.0mm]south},end anchor={[yshift=+0mm]east}]
  \arrow[dll,bend left=8,start anchor={[xshift=+1.5mm]south},end anchor={[yshift=-0.5mm]east}]\\
  & & X_s(8)
\end{tikzcd}
$$
In what follows we will show that the supersingular points split completely over $\FF_{p^2}$
on the quotients of $X(16,8,16)$ not covering $X(8)$, for every odd prime $p$.

\subsection*{Supersingular fields of definition}

\begin{theorem}\label{thm:ss_moduli_field_X0}
For any positive integer $N$, the supersingular invariants on the modular curve $X_0(N)$ are defined over\space\space$\FF_{p^2}$, and if $p \equiv \pm1 \bmod N$, then the supersingular invariants also split over\space\space$\FF_{p^2}$ on $X_1(N)$.
\end{theorem}

\begin{proof}
  For any elliptic curve $E$ in the isogeny class of a curve over $\FF_p$,
  the full endomorphism ring $\cO$ is defined over $\FF_{p^2}$.
  Since the action of $\cO/N\cO \isom \MM_2(\ZZ/N\ZZ)$ on the $E[N]$ is
  defined over $\FF_{p^2}$, it follows that the Galois action on $E[N]$,
  which commutes with $\cO/N\cO$, acts through the center $(\ZZ/N\ZZ)^*$
  of $\GL_2(\ZZ/N\ZZ)$, and more precisely, Frobenius acts as $-p$ on $E[N]$.
  Consequently, the lines are Galois stable and every cyclic $N$-isogeny
  is defined over $\FF_{p^2}$. In view of the action of Frobenius, if
  $p \equiv \pm1 \bmod N$, the Galois action on the $N$-torsion of $E$
  or its twist is trivial, so the supersingular moduli are defined in
  $\FF_{p^2}$.
\end{proof}

\noindent{\bf Remark.} Equivalently, for $X_0(N)$ we can state that every
supersingular $j$-invariant $j_0$ splits completely under the map
$X_0(N) \to X(1)$, or that the polynomial $\Phi_N(x,j_0)$ splits completely,
where $\Phi_N(x,y)$ is the classical modular polynomial.
For $X_1(N)$, the splitting of the supersingular points is recognized
by the factorization of the $N$-division polynomial $\psi_N$.
\vspace{1mm}

As a consequence, the split Cartan modular curve $X_s(N)$, defined by the
congruence subgroup
$$
\Gamma_s(N) = \left\{\;
  \left(\begin{array}{@{}cc@{}}a&b\\c&d\end{array}\right)
  \;:\; b \equiv c \equiv 0 \bmod N \right\},
$$
parametrizing elliptic curves with a disjoint pair of cyclic $N$-isogenies,
also splits the supersingular moduli.

\begin{corollary}
  \label{cor:ss_moduli_field_split_Cartan}
  For any positive integer $N$, then the supersingular invariants on the
  split Cartan modular curve $X_s(N)$ are defined over $\FF_{p^2}$.
  In particular if $p \equiv \pm1 \bmod N$, then the supersingular invariants
  on the modular curve $X(N)$ are defined over $\FF_{p^2}$.
\end{corollary}

\begin{proof}
  The first statement follows from the splitting of $N$-isogenies over $\FF_{p^2}$.
  In addition if $p \equiv \pm1 \bmod N$, the points of each kernel are fixed,
  hence a basis is defined over $\FF_{p^2}$ (up to twist).
\end{proof}

\noindent{\bf Remark.} For the levels $N$ in $\{1,2,3,4,6\}$, the unit
group $(\ZZ/N\ZZ)^*/\{\pm1\}$ is trivial so the supersingular points
split for all $p$. This corresponds to the geometric equalities $X_1(N) = X_0(N)$
and $X(N) = X_s(N)$.
\vspace{2mm}

The Weber moduli are functions on $X(48)$ which map through $\cW_{24}$.
To show the splitting of supersingular points on $\cW_{24}$ it suffices to prove
it for $\cW_3$ and $\cW_8$. However, $X(6)$ covers $\cW_3$, so the supersingular
moduli on $\cW_3$ split over $\FF_{p^2}$ by the previous theorem. To prove that
they split on $\cW_8$ it is necessary to consider the factorization
$$
\begin{tikzcd}
  X(16,8,16) \arrow[d,"\big\langle\binom{\,13\;8\,}{\,\;8\;\;5\,}\big\rangle" swap]\arrow[r]
  & X(8) \arrow[d,"\big\langle\binom{\,5\;0\,}{\,0\;5\,}\big\rangle"]\\
  \cW_8 \arrow[r] & X_s(8),
\end{tikzcd}
$$
where $X(16,8,16)$ is the quotient of $X(16)$ by the diagonal matrix group
$\langle\pm 9I_2\rangle \subset \SL_2(\ZZ/16\ZZ)/\{\pm1\}$.

The supersingular points split in $X_s(8)$ by the previous theorem.
On the other hand, for the classes $p \bmod 8$ in the coset
$\{ \pm5 \} \subset (\ZZ/8\ZZ)^*/\{\pm1\}$ form an obstruction to lifting
supersingular points to $X(8)$ over $\FF_{p^2}$.
Clearly, since $\langle{9I_2}\rangle \subset \Gamma(16,8,16)/\Gamma(16)$,
for the primes $p$ such that $p \bmod 16$ lie in the kernel
$$
\langle{-1,9}\rangle = \{\pm1,\pm9\} \subset (\ZZ/16)^*/\{\pm1\} \longrightarrow (\ZZ/8)^*/\{\pm1\},
$$
the supersingular invariants in $X(16,8,16)$ split over $\FF_{p^2}$.  It remains
to show that the obstruction vanishes also on the coset $\{\pm3,\pm5\}$.  However, this
follows since the subgroup of $\Gamma_8/\Gamma(16)$ surjects on the diagonal subgroup
of $\Gamma_s(8)/\Gamma(8)$:
$$
\left\langle\left(\begin{array}{@{}cc@{}}13&8\\8&5\end{array}\right)\right\rangle
  \longrightarrow
\left\langle\left(\begin{array}{@{}cc@{}}5&0\\0&5\end{array}\right)\right\rangle
$$
under $\SL_2(\ZZ/16\ZZ) \to \SL_2(\ZZ/8\ZZ)$, corresponding to the fact that $\cW_8$
does not factor through $X(8)$.  This establishes the following theorem.

\begin{theorem}
  \label{thm:weber_field}
  The supersingular Weber invariants on $\cW_{24}$ are defined over $\FF_{p^2}$.
\end{theorem}

\noindent{\bf Remark.} A point $(u_0,u_1,u_2)$ on $\cW_{24}$ over the $j$-invariant
$j_0$ consists of a triple of common roots of the polynomial
$
(x^{24}-16)^3 - j_0 x^{24},
$
and the set of roots is precisely $\{\, \zeta_{24}^i u_j \;:\; 0 \le i < 24, 0 \le j < 3\, \}$.
The property that $j_0$ splits completely under $\cW_{24} \to X(1)$ over $\FF_{p^2}$
is equivalent to this polynomial splitting completely over $\FF_{p^2}$.

\section*{Elliptic curves over Weber and Fermat curves}
\label{EllipticSurfaces}

We now turn to the problem of explicit families of elliptic curves over Weber curves
$\cW_n$ and Fermat curves $\cF_n$ and the isogeny structures that they parametrize.
We construct models over the quotients $\cW_n \longrightarrow \cX_n$ and
$\cF_n \longrightarrow \cY_n$.
For each $n$ dividing $8$, we have an isomorphism $\cW_n \isom \cF_n$ (given explicitly
in the Appendix), but this does not imply a morphism between $\cX_n$ and $\cY_n$,
except for $n = 1$, we have seen in~\eqref{eqn:Weber1-Fermat1-isom} the isomorphisms
$$
\begin{tikzcd}
  \cW_1 \arrow[r,"\isom"] & \cF_1 \arrow[r,"\isom"] & \cY_1,
\end{tikzcd}
$$
which implies a morphism $X(2) \isom \cY_1 \longrightarrow \cX_1 \isom X_0(2)$.

We have also seen that the family~\eqref{eqn:E0/Xn} of elliptic curves
$$
\cE_0/\cX_n: y^2 = x\left(x^2 - \frac{(u^n - 64)}{4}x - (u^n - 64)\right)\ccomma
$$
is parametrized by the Weber modular curve $\cX_n$, hence by base extension by $\cW_n$.
We also recall from~\eqref{eqn:C0/Y1-C1/Y1} that the affine Fermat curve
$\cF_n : s^n + t^n + 1 = 0$ parametrizes the following model through its quotient
to $\cY_n$:
$$
\cC_0/\cY_n:y^2 = x(x - 1)(x + t^n) = x(x - 1)(x - s^n - 1),
$$
In order to set up the notation for the extension of base field, from $K(\cY_1)$ to
$K(\cY_8)$, we set $t_3 = t$ a generator for $K(\cY_8)$ and $t_{k-1} = t_{k}^2$
for $1 \le k \le 3$, such that for a field $K$,
$$
\begin{tikzcd}
  K(\cF_1) \isom K(\cW_1) \arrow[d]\arrow[r,hook] &
  K(\cF_2) \isom K(\cW_2) \arrow[d]\arrow[r,hook] &
  K(\cF_4) \isom K(\cW_4) \arrow[d]\arrow[r,hook] &
  K(\cF_8) \isom K(\cW_8) \arrow[d] \\
  K(\cY_1) = K(t_0) \arrow[r,hook] & K(\cY_2) = K(t_1) \arrow[r,hook] & K(\cY_4) = K(t_2) \arrow[r,hook] & K(\cY_{8}) = K(t_3)
\end{tikzcd}
$$
is a sequence of quadratic extensions $K(t_k)/K(t_{k-1})$, which will permit us
to build a chain of $2$-isogenies.  We note that the final isomorphism
$K(\cF_8) \isom K(\cW_8)$ requires a square root of $2$ in $K$, and moreover
the chain of isogenies which follow require an $8$-th root of unity $\zeta_8$,
for which we set $i = \zeta_8^2$, so we assume $K$ contains a homomorphic
image of $\ZZ[\zeta_8]$.

\subsection*{Explicit $2$-isogeny chains}
We prioritize the Fermat model as the base curve parametrizing chains of $2$-isogenies,
with a view to constructing explicit equations for the isogeny chains from $\cC_0$.
We describe a $2$-isogeny chain over $K(\cY_8)$, starting from $\cC_0$, of the form
$$
\begin{tikzcd}
  \cC_0 \arrow[r,"\phi_0"] & \cC_1 \arrow[r,"\phi_1"] & \cC_2 \arrow[r,"\phi_2"] & \cC_3.
\end{tikzcd}
$$
The base extension to $K(\cF_8)$ gives rise to an explicit action on the chains
by the automorphism group of $K(\cF_8)/K(X(1))$, which,
by Proposition~\ref{prop:Weber_groups}, is isomorphic to $S_3 \ltimes \nabla(\mmu_8^2)$.

We normalize the curves $\cC_i$ and isogenies $\phi_i$ of the chain
such that the isogenies are defined successive quotients by a $2$-torsion
point $(0,0)$.  
On the curve $\cC_k$ we give, for each $k$ in $\{0,1\}$, a point $T_k \in \cC_k[4]$
such that $2T_k = (0,0)$ and $\phi_k(T_k) = (0,0)$ on the codomain $\cC_{k+1}$.
The quotient by $(0,0)$ in $\cC_0[2]$ gives the $2$-isogeny:
$$
\begin{tikzcd}[row sep=0mm]
  \phi_0: \cC_0: y^2 = x(x - 1)(x + t_0) \arrow[r] & \cC_1: y^2 = x(x + 4c_0)(x + e_0^2)\\
  (x,y) \arrow[r,mapsto] & \dsp\left(\frac{(x-c_0)^2}{x}, \frac{x^2-c_0^2}{x^2}y\right)
\end{tikzcd}
$$
where $c_0 = it_1$ and $e_0 = t_1+i$.
The point $T_0 = (c_0,c_0 e_0) \in \cC_0[4]$ satisfies $2T_0 = (0,0) \in \cC_0[2]$
and $\phi_0(T_0) = (0,0) \in \cC_1[2]$.
The quotient by $(0,0)$ in $\cC_1[2]$ gives the next step in the $2$-isogeny chain:
$$
\begin{tikzcd}[row sep=0mm]
  \phi_1: \cC_1: y^2 = x(x + 4c_0)(x + e_0^2) \arrow[r] & \cC_2: y^2 = x(x + 4c_1))(x + e_1^2)\\
  (x,y) \arrow[r,mapsto] & \dsp\left(\frac{(x-c_1)^2}{x}, \frac{x^2-c_1^2}{x^2}y\right)
\end{tikzcd}
$$
where $c_1 = 2\zeta_8t_2(t_1+i)$ and $e_1 = (t_2 + \zeta_8)^2$.
The point $T_1 = (c_1, c_1 e_1)$ satisfies $2T_1 = (0,0) \in \cC_1[2]$ and
$\phi_1(T_1) = (0,0) \in \cC_2[2]$.
Finally, the quotient by $(0,0)$ in $\cC_2[2]$ gives the $2$-isogeny:
$$
\begin{tikzcd}[row sep=0mm]
  \phi_2: \cC_2: y^2 = x(x + 4c_1)(x + e_1^2) \arrow[r] & \cC_3: y^2 = (x^2 + 4c_2)(x + c_3)\\
  (x,y) \arrow[r,mapsto] & \dsp\left(\frac{x^2-c_2}{x}, \frac{x^2+c_2}{x^2}y\right)
\end{tikzcd}
$$
where $c_2 = -4c_1e_1^2$ and $c_3 = e_1^2 + 4c_1$.
In terms of the coordinates $(s_3,t_3) = (s_3:t_3:1)$,
the coordinate permutations and scalar multiplications
$(s_3,t_3) \mapsto (\zeta_8^{i_1} s_3,\zeta_8^{i_2} t_3)$
permute all possible $2$-isogeny chains from $\cC_0$.

\subsection*{Explicit $\ell$-isogeny chains}

Unlike for $\ell = 2$, the Weber functions do not parametrize a
$\Gamma_0(\ell)$-structure for any odd prime $\ell$.
Even on $\cW_{3n}$, the nonsplit Cartan structure $\Gamma_{ns}^+(3)$
determined by the Weber curve is independent of a $\Gamma_0(3)$
structure parametrizing $3$-isogenies.
An $\ell$-isogeny chain is parametrized by a sequence of moduli
$(u_0,u_1,\dots,u_r)$ which are successive roots
$\Phi_\ell(u_{i-1},u_i) = 0$ of the level-$\ell$ modular polynomial
$\Phi_\ell(x,y)$ with respect to $\cX_{24}$
(or when $\ell = 3$, of one of the quotients, $\cX_8$ or $\cY_8$
of $\cW_8 \isom \cF_8$).

\section{Weber supersingular modules}

Mestre's method of graphs~\cite{Mestre86} permits one
to construct a {\it supersingular Hecke module} -- a free abelian
group on supersingular points equipped with Hecke operators,
given by correspondences on the underlying modular curves.
In terms of the supersingular isogeny graphs, the Hecke
operator $T_\ell$ arises as the adjacency matrices of
the $\ell$-isogeny graph.
Mestre's construction was defined for modular curves of level~1,
or one of the models $X_0(m)$ of genus $0$ given by quotients
of the Dedekind $\eta$ function, for $m \in \{2,3,4,5,7,13\}$.
This gives rise to Galois representations of type~$\GL_2$
and level $N = mp$.  As an application, Cowan~\cite{Cowan22}
exploits Mestre's method with sieving to enumerate
newforms associated to low dimensional isogeny factors of
the Jacobian $J_0(N)$ of $X_0(N)$.
In particular, this permits one to gather statistics for
the distribution of elliptic curves and low dimensional
modular abelian varieties of prime level, or nearly prime
level.

The Weber and Fermat modular curves permit one to study modular
isogeny factors of levels with higher powers of $2$ and $3$.
A general framework for isogeny graphs with level structure
is developed in~\cite{ModSS}.
For prime level, Bennett, Gherga, and Rechnitzer~\cite{BGR19}
have enumerated all elliptic curves of prime conductor up
to $2 \cdot 10^9$, but datasets for elliptic curves of composite
level are less complete.  In particular, the higher powers
of~$2$ and~$3$ imply that the supersingular modules associated
to moderately large primes rapidly exceed conductor bounds in
typical datasets such as the LMFDB~\cite{LMFDB}.
In the following table, we give the numbers of isogeny classes
of pseudo elliptic factors coming from the Weber curves $\cX_n$,
where $n$ divides $12$, or Fermat curves $\cY_n$, where $n$
divides $8$.  By pseudo elliptic factors, we refer to an
eigenspace on which the Hecke operators away from $6p$ act
by scalar multiplication in $\ZZ$.

\begin{table}[h!]
\caption{Counts of Weber and Fermat pseudo elliptic factors}
\begin{tabular}{@{}r|*{10}{c|}}
\multicolumn{2}{c|@{}}{$B-1000 < p < B$}
                     & 1000 & 2000 & 3000 & 4000 & 5000 & 6000 & 7000 & 8000 & $N = mp$ \\ \hline
  \# Weber :    & \multirow{2}{*}{$\cX_{12}$}
                     &  160 &  104 &   80 &  132 &   92 &   60 &   84 &   68 & \multirow{2}{*}{$288 \cdot p$} \\
  \# orbits :   &    &   40 &   26 &   20 &   22 &   23 &   15 &   21 &   17 & \\ \hline
  \# Weber :    & \multirow{2}{*}{$\cX_6$}
                     &  212 &  132 &   98 &   78 &  106 &   92 &   80 &   76 & \multirow{2}{*}{$72 \cdot p$} \\
  \# orbits :   &    &  106 &   66 &   49 &   39 &   53 &   46 &   40 &   38 & \\ \hline
  \# Weber :    & \multirow{2}{*}{$\cX_3$}
                     &  450 &  288 &  222 &  212 &  262 &  208 &  206 &  194 & \multirow{2}{*}{$18 \cdot p$} \\
  \# orbits :   &    &  225 &  144 &  111 &  106 &  131 &  104 &  103 &   97 & \\ \hline
  \# Weber :    & \multirow{2}{*}{$\cX_4$}
                     &  142 &  110 &   80 &   66 &   48 &   60 &   56 &   56 & \multirow{2}{*}{$32 \cdot p$} \\
  \# orbits :   &    &   71 &   55 &   40 &   33 &   24 &   30 &   28 &   28 & \\ \hline
  \# Weber :    & $\cX_2$
                     &  174 &   99 &   91 &   79 &   94 &   84 &   90 &   63 & $8 \cdot p$ \\ \hline
  \# Weber :    & $\cX_1$
                     &  274 &  256 &  184 &  194 &  175 &  186 &  155 &  126 & $2 \cdot p$ \\ \hline
  \# Fermat :   & \multirow{2}{*}{$\cY_8$}
                     &  100 &   60 &   32 &   32 &   72 &   24 &   24 &   24 & \multirow{2}{*}{$256 \cdot p$} \\
  \# orbits :   &    &   25 &   15 &    8 &    8 &   18 &    6 &    6 &    6 & \\ \hline
  \# Fermat :   &  \multirow{2}{*}{$\cY_4$}
                     &  974 &  538 &  538 &  544 &  502 &  498 &  396 &  336 & \multirow{2}{*}{$64 \cdot p$} \\
  \# orbits :   &    &  487 &  269 &  269 &  272 &  251 &  249 &  298 &  168 & \\ \hline
  \# Fermat :   & $\cY_2$
                     &  545 &  364 &  309 &  287 &  281 &  279 &  248 &  193 & $16 \cdot p$ \\ \hline
  \# Fermat :   & $\cY_1$
                     &  96  &   60 &   48 &   39 &   37 &  39  &   34 &   26 & $4 \cdot p$ \\ \hline
  \# Level 1 :  & $X(1)$
                     &   69 &   41 &  34  &   35 &   25 &   26 &   21 &   22 & $p$ \\ \hline
  \# primes :   & $\ZZ$
                     & 168  &  135 &  127 &  120 &  119 &  114 &  117 &  107 & $1$ \\  \hline
\end{tabular}
\label{table:pseudo-elliptic}
\end{table}

Table~\ref{table:pseudo-elliptic} collects the counts of pseudo elliptic factors
associated to the supersingular Hecke modules on the Weber curves $\cX_n$, for
$n$ dividing $12$ and the Fermat curves $\cF_n$, for $n$ dividing $8$.
When the factors appear in orbits of twists with respect to $\QQ(\zeta_n)$,
we report in a second line the number of such orbits (dividing the counts
by $2$ or $4$).

The full Weber curve $\cW_{24}$ admits a level $48$, equipped with an action
of Galois group of $\QQ(\zeta_{48})$, which is not of exponent $2$, and we
observe that every modular isogeny factor associated to the Weber curves
$\cX_{24}$ and $\cX_{8}$, has even degree Hecke algebra, and in particular,
gives rise to no pseudo elliptic factors.  More specifically, the Hecke
algebra systematically contains the quadratic field $\QQ(\sqrt{2})$.

Although this represents a limited dataset, the numbers of such orbits for
$\cX_{12}$, of level $288p$, aligns roughly with the numbers of isogeny
classes of elliptic curves of conductor $p$, while the numbers of orbits
for $\cY_{8}$, of level $256p$, appears to be smaller by a factor of four
in this range.
Each of the curves of lower levels give rise to a much larger number of
pseudo elliptic factors of levels $mp$ compared to the numbers of elliptic
curves of prime level $p$, for $p$ in the same range.  The vast majority
of the pseudo elliptic factors are in fact isogeny classes of modular
elliptic curves.  The high powers of $2$ and $3$ in the cofactor $m$
imply that we are able to enumerate data for isogeny classes of elliptic
curves of conductor exceeding typical bounds on the level in datasets
such as the LMFDB~\cite{LMFDB}.  In particular the levels attained
for primes $p$ up to $8000$ gives levels $288p$ or $256p$ exceeding
$2 \times 10^6$, which is beyond what has been systematically treated
for composite levels.

\section{Conclusion}

The Weber modular curves $\cW_{24}$, and their Fermat modular quotients $\cF_8$,
present highly symmetric models for modular curves of level $48$ and $16$, whose
quotients $\cX_{24}$ and $\cY_{8}$ give high-degree covers of the $j$-line.
This high degree and the symmetry properties make them practical for computation
of isogenies.  In particular the small coefficient size and sparseness provide
both elegant and efficient isogeny relations.  While not all curves admit a Weber
structure, we show that the supersingular points on $\cW_{24}$ are all defined
over $\FF_{p^2}$. This makes the Weber curves interesting for the local study
and navigation of supersingular isogeny graphs, with application to modern
isogeny-based cryptographic protocols.  On the other hand, the $\ell$-isogeny
graphs also play an important role in the theory of modular forms and Galois
representations, coming from supersingular Hecke modules.
In particular, sieving for the pseudo elliptic factors of such modules shows
that one can compute data of elliptic modular curves of conductor beyond what
can be systematically computed for arbitrary level, and can be used to
understand the distribution of such curves with admitting multiplicative
reduction at a large prime and whose conductor is simultaneously divisible
by a large power of~$2$ and~$3$.

\section*{Appendix}

\section*{Isomorphisms between Weber and Fermat curves}
  
The isomorphisms $\cW_n \longrightarrow \cF_n$ and their inverses $\cF_n \longrightarrow \cW_n$
are given as follows:
$$
\begin{tikzcd}[row sep=-2mm,column sep=8mm]
  \cW_8: \left\{
  \begin{array}{@{}c@{}}
    X_0^8 + X_1^8 + X_2^8 = 48 X_3^8,\\
    X_0 X_1 X_2 = \sqrt{8} X_3^3
  \end{array} \right. \arrow[r] & \cF_8 : X^8 + Y^8 + Z^8 = 0\\[1mm]
  (u_0:u_1:u_2:u_3) \arrow[r,mapsto] &
  (\sqrt{2} u_2^5 u_3 - u_0^3 u_1^3 : u_1^6 - 2 u_0^2 u_2^2 u_3^2 : \sqrt{2} u_0^5 u_3 - u_1^3 u_2^3)\\[1mm]
  (\sqrt{2}s_0^3:\sqrt{2}s_1^3:\sqrt{2}s_2^3:s_0s_1s_2) & \arrow[l,mapsto] (s_0:s_1:s_2)
\end{tikzcd}
$$
$$
\begin{tikzcd}[row sep=-2mm,column sep=8mm]
  \cW_4: \left\{
  \begin{array}{@{}c@{}}
    X_0^4 + X_1^4 + X_2^4 = 48 X_3^4,\\
    X_0 X_1 X_2 = 8 X_3^3
  \end{array} \right. \arrow[r] & \cF_4 : X^4 + Y^4 + Z^4 = 0\\[1mm]
  (u_0:u_1:u_2:u_3) \arrow[r,mapsto] & (u_0^3 - 2u_1 u_2 u_3 : u_1^3 - 2u_0 u_2 u_3 : u_2^3 - 2u_0 u_1 u_3)\\[1mm]
  (2s_0^3:2s_1^3:2s_2^3:s_0s_1s_2) & \arrow[l,mapsto] (s_0:s_1:s_2)
\end{tikzcd}
$$
$$
\begin{tikzcd}[row sep=-2mm,column sep=8mm]
  \cW_2: \left\{
  \begin{array}{@{}c@{}}
    X_0^2 + X_1^2 + X_2^2 = 48 X_3^2,\\
    X_0 X_1 X_2 = 64 X_3^3
  \end{array} \right. \arrow[r] & \cF_2 : X^2 + Y^2 + Z^2 = 0\\[1mm]
  (u_0:u_1:u_2:u_3) \arrow[r,mapsto] &
  \left\{\begin{array}{@{\;}l}
  (-u_0^2 + 16 u_3^2: u_0 u_1 - 4 u_2 u_3: u_0 u_2 - 4 u_2 u_3)\\
  (u_0 u_1 - 4 u_2 u_3: -u_1^2 - 16 u_3^2: u_1 u_2 - 4 u_0 u_3)\\
  (u_0 u_2 - 4 u_1 u_3: u_1 u_2 - 4 u_0 u_3: -u_2^2 + 16 u_3^2)
  \end{array}\right.\\[1mm]
  (4s_0^3:4s_1^3:4s_2^3:s_0s_1s_2) & \arrow[l,mapsto] (s_0:s_1:s_2)
\end{tikzcd}
$$
$$
\begin{tikzcd}[row sep=-2mm,column sep=8mm]
  \cW_1: \left\{
  \begin{array}{@{}c@{}}
    X_0 + X_1 + X_2 = 48 X_3^2,\\
    X_0 X_1 X_2 = 4096 X_3^3
  \end{array} \right. \arrow[r] & \cF_1 : X + Y + Z = 0\\[1mm]
  (u_0:u_1:u_2:u_3) \arrow[r,mapsto] & (u_0-16u_3 : u_1-16u_3 : u_2-16u_3)\\[1mm]
  (16s_0^3:16s_1^3:16s_2^3:s_0s_1s_2) & \arrow[l,mapsto] (s_0:s_1:s_2)
\end{tikzcd}
$$

\section*{Elliptic curves over twisted Fermat curves}

Let $\cF_n^t: s^n + t^n = 1$ be the twist by $\zeta_{2n}$ of the Fermat curve
$\cF_n: s^n + t^n + 1 = 0$.  The symmetry of the Fermat modular curve simplifies
the study of automophisms, but the $2$-isogenies extend more naturally along
a prescribed chain by breaking the $S_3$-symmetry by this twist.
We revisit the $2$-isogeny chains of Section~\ref{EllipticSurfaces} in terms
of this twisted Fermat curve, in relation to the following $-1$-twist of the
family of elliptic curves:
$$
\cC_0^t/\cF_n^t: y^2 = x(x + 1)(x + t^n).
$$
As above, we set $t = t_0$ and $t_{i-1} = t_{i}^2$ for $0 \le i \le 3$, such
that for a field $K$,
$$
K(\cY_1) = K(t_0) \subset K(\cY_2) = K(t_1) \subset K(\cY_4) = K(t_2) \subset K(\cY_{8}) = K(t_3)
$$
is a sequence of quadratic extensions (relative to the previous functions
$t_i$, these generators are scaled by $\zeta_{2n}^{2^i}$).
Over $K(\cY_8)$ we can define a chain of $2$-isogenies starting from $\cC_0$,
$$
\begin{tikzcd}
  \cC_0^t \arrow[r,"\phi_0"] & \cC_1^t \arrow[r,"\phi_1"] & \cC_2^t \arrow[r,"\phi_2"] & \cC_3^t
\end{tikzcd}
$$
normalized such that the isogenies are defined successive quotients by a $2$-torsion
point $(0,0)$ gives cyclic extension over the base curve $\cY_8$.
On $\cC_k$ we give $T_k$, for $k$ in $\{0,1\}$ such that $2T_k = (0,0)$ and $\phi_k(T_k) = (0,0)$
on the codomain $\cC_{k+1}$.
$$
\begin{tikzcd}[row sep=0mm]
  \phi_0: \cC_0^t: y^2 = x(x + 1)(x + t_0) \arrow[r] & \cC_1^t: y^2 = x(x + 4c_0)(x + e_0^2)\\
  (x,y) \arrow[r,mapsto] & \dsp\left(\frac{(x-c_0)^2}{x}, \frac{x^2-c_0^2}{x^2}y\right)
\end{tikzcd}
$$
where $c_0 = t_1$ and $e_0 = t_1+1$. 
The point $T_0 = (c_0,c_0 e_0)$ satisfies $2T_0 = (0,0) \in \cC_0^t[2]$ and $\phi_0(T_0) = (0,0) \in \cC_1^t[2]$.
Moreover, the image of $\cC_0^t[2]$ is generated by
$
\phi_0\big((1,0)\big) = \phi_0\big((t_1^2,0)\big) = \big((t_1-1)^2,0\big).
$
$$
\begin{tikzcd}[row sep=0mm]
  \phi_1: \cC_1^t: y^2 = x(x + 4c_0)(x + e_0^2) \arrow[r] & \cC_2^t: y^2 = x(x + 4c_1)(x + e_1^2)\\
  (x,y) \arrow[r,mapsto] & \dsp\left(\frac{(x-c_1)^2}{x}, \frac{x^2-c_1^2}{x^2}y\right)
\end{tikzcd}
$$
where $c_1 = 2t_2e_0$ and $e_1 = (t_2+1)^2$. The point $T_1 = (c_1, c_1e_1^2)$ satisfies $2T_1 = (0,0) \in \cC_1^t[2]$
and $\phi_1(T_1) = (0,0) \in \cC_2^t[2]$.
Finally, the quotient by $(0,0)$ in $\cC_2^t$ gives the $2$-isogeny:
$$
\begin{tikzcd}[row sep=0mm]
  \phi_2: \cC_2^t: y^2 = x(x + 4c_1)(x + e_1^2) \arrow[r] & \cC_3^t: y^2 = x(x + 4c_2)(x + e_2^2)
  \\
  (x,y) \arrow[r,mapsto] & \dsp\left(\frac{x^2+c_2^2}{x}, \frac{x^2-c_2^2}{x^2}y\right)
\end{tikzcd}
$$
where $c_2 = u_3e_1^2$ and $e_2 = t_3 u_3 + e_1$ where $u_3^2 = 8e_0$.

After renormalization, these curves define the sequence of Legendre invariants:
$$
\big(\lambda_0,\lambda_1,\lambda_2,\lambda_3\big)
= \left(t_0, \frac{e_0^2}{4c_0}, \frac{e_1^2}{4c_1}, \frac{e_2^2}{4c_2} \right)\ccomma
$$
satisfying the modular equation
$$
\lambda_{i+1}(\lambda_{i+1} - 1) = \frac{(\lambda_i - 1)^2}{16\lambda_i}.
$$


\end{document}